\documentclass[11pt]{amsart}
	
	\usepackage{amsmath}		%
	\usepackage{amssymb}		%
	\usepackage{amscd}		%
	\usepackage{amsfonts}	%
	\usepackage{amsthm}
	\usepackage{upref}

\newtheorem{theorem}{Theorem}[section]
\newtheorem{corollary}[theorem]{Corollary}
\newtheorem{proposition}[theorem]{Proposition}
\newtheorem{lemma}[theorem]{Lemma}

\theoremstyle{remark}
\newtheorem{remark}[theorem]{Remark}

\theoremstyle{definition}

\numberwithin{equation}{section}

\newcommand{\lie}[1]{\mathfrak{#1}}
\newcommand{\Nm}[1]{\left\|{#1}\right\|}
\newcommand{\NM}[1]{\left|\!\left|\!\left|{#1}\right|\!\right|\!\right|}

\newcommand{\Cent}{Z{}}
\newcommand{\GL}{\mathrm{GL}{}}

\renewcommand{\d}{{\mathbf{d}}}

\newcommand{\C}{{\mathbf{C}}}
\newcommand{\Id}{\mathrm{Id}}
\newcommand{\Hom}{\mathrm{Hom}}
\newcommand{\End}{\mathrm{End}}

\newcommand{\tens}{\otimes}

\newcommand{\R}{\mathbf{R}}
\newcommand{\Q}{\mathbf{Q}}
\newcommand{\N}{\mathbf{N}}
\newcommand{\Tr}{\mathbf{Tr}}

\newcommand{\Ad}{\mathrm{Ad}}
\newcommand{\ad}{\mathrm{ad}}

\newcommand{\cf}{cf.}

\newcommand{\abs}[1]{\mathopen|#1\mathclose|}

\newcommand{\sur}{/}

\newcommand{\infa}{<}

\newcommand{\url}[1]{\texttt{#1}}
\newcommand{\Zar}{\mathrm{Z}}
\newcommand{\Zcl}{\mathrm{Zcl}}

\begin{document}

\title[Linear actions of reductive groups and homogenous dynamics]{Geometric results on linear 
actions of reductive Lie groups for applications to homogeneous dynamics}
\author[R. Richard]{Rodolphe Richard}
\address{University College London, London - WC1E 6BT, U.K.}
\email{Rodolphe.RICHARD@Normalesup.org}

\author[N. Shah]{Nimish A. Shah}
\address{The Ohio State University, Columbus, OH 43210, U.S.A.}
\thanks{N.S. acknowledges support of NSF grants.}
\email{shah@math.ohio-state.edu}



\begin{abstract}
Several problems in number theory when reformulated in terms of homogenous dynamics 
involve study of limiting distributions of translates of algebraically defined measures on orbits of 
reductive groups. The general non-divergence and linearization techniques, in view of Ratner's 
measure classification for unipotent flows, reduce such problems to dynamical questions about 
linear actions of reductive groups on finite dimensional vectors spaces. This article provides 
general results which resolve these linear dynamical questions in terms of natural group theoretic or geometric conditions. 
\end{abstract}

\maketitle

\section{Introduction}
Many questions in number theory that involve two interacting groups of symmetries contained in 
a Lie group can be addressed using methods of homogeneous dynamics. One such class of 
problems was proposed by Duke, Rudnick and Sarnak~\cite{DRS}, where one wants to study the 
density of integer points on affine varieties defined over $\Q$ which admit transitive action of a 
semisimple Lie group. In this case one symmetry group is an arithmetic lattice which preserves 
the integer points and the other being a reductive group, which is the stabilizer of a rational point 
in the variety. In such a situation, due to a finiteness result of Borel and Harish-Chandra, the 
question reduces to considering density of points on discrete orbits of lattices in semisimple Lie 
groups with reductive stabilizers. The approach suggested by \cite{DRS} involves understanding 
the limit distributions of translates of closed orbits of reductive subgroups on finite volume (or 
periodic)  homogeneous spaces of semisimple groups by various sequences of elements of the 
semisimple group. Finding the precise algebraic relation between the translating sequences, the 
reductive subgroup, and the limiting distributions is a key to this method.

In \cite{DRS} it was shown that if $H$ is a symmetric subgroup of a semisimple Lie groups $G$ 
with an irreducible lattice $\Gamma$ which also intersects $H$ in a lattice, then for any 
sequence $y_{n}\to\infty$ in $G$ modulo $H$, one has that the sequence $y_{n}\mu_{H}$ gets 
equidistributed in $G/\Gamma$ (with respect to a $G$-invariant probability measure), where $
\mu_{H}$ denotes the $H$-invariant probability measure on 
$H/H\cap\Gamma\hookrightarrow G/\Gamma$. 
The result was proved using deep results on spectral theory of $L^{2}(G/\Gamma)$. 
Later the same conclusion was obtained by Eskin and McMullen~\cite{EM} using the mixing 
property of certain subgroup actions on homogeneous spaces. 

In the general case, when $H$ is not a symmetric subgroup, the limit distributions of the 
sequences $y_{n}\mu_{H}$ depend on algebraic or geometric relations between $H$ and the 
sequence $y_{n}$. This relation was analyzed in fairly general situation by Eskin, Mozes and 
Shah~\cite{EMSGAFA,EMS}, where a conditional answer to a question in \cite{DRS} on the 
density of integer points was obtained. The main steps of this technique are as follows: (1) find 
the condition under which the sequence $\{y_{n}\mu_{H}\}$ is pre-compact in the space of 
probability measures on $G/\Gamma$ with respect to the weak$^{\ast}$ topology; (2) showing 
that any of the accumulation points, which are colloquially referred to as limiting measures, of 
this sequence of measures is invariant under a nontrivial unipotent subgroup; (3) apply Ratner's 
Theorem~\cite{R-measure} classifying ergodic invariant measures of unipotent flows to analyze 
the limiting measures using the linearization technique developed in 
\cite{DS,DM-MathAnn,Shah-uniform,DM-Limit}. Due to Ratner's measure classification theorem, 
this type of linearization technique reduces all the three steps into geometric questions about linear representations. The purpose of this article is to develop new techniques 
and obtain results to answer these linear dynamical questions. These results are fundamental for further developments on the above program. For example, 
Richard and Zamojski~\cite{focusing}) have vastly generalized and enhanced the core theorems of 
\cite{EMSGAFA,EMS} using the main results of this article for non-arithmetic situations and also for 
$S$-arithmetic situations. 

\subsection{Terminological conventions}\label{conventions}
All Lie groups are assumed to be \emph{finite dimensional real Lie groups}. By a connected 
\emph{reductive subgroup $H$ in a semisimple Lie group $G$}, we mean a closed and 
connected subgroup whose Lie algebra $\lie{h}$ is a \emph{reductive subalgebra} in the Lie 
algebra $\lie{g}$ of $G$, according to \cite[\S6 N$^{0}$6~Def.~5]{B}. Namely we ask for the 
adjoint action of $\lie{h}$ on $\lie{g}$ to be semisimple. We record three equivalent criteria 
for $H$ to be reductive in $G$: 
\renewcommand{\labelenumi}{(\alph{enumi})}
\begin{enumerate}
 \item the radical of $H$ does not contain non-trivial unipotent elements of $G$ 
 (\cf~\cite[\S6 N$^{0}$5~Theorem~4]{B}) and hence $H$, having central radical, is reductive;
 \item every representation of $H$ induced by a finite dimensional linear representation of $G$ is  
 semisimple~\cite[\S6 N$^{0}$6~Corollary~1 and \S6 N$^{0}$2~Theorem~2]{B};
 \item $H$~is stable under at least one global Cartan involution of $G$ (see \cite{Mostowherm} 
and section~\ref{Cartan}). 
\end{enumerate}

\subsection{Statements of the main results}\label{application}
In this article we prove the following: 

\begin{theorem}[Linearized Non-divergence] \label{lemma}
 Let $G$ be a connected semi\-simple real Lie group and let $H$ be a connected 
\emph{reductive subgroup in $G$}. 
 Let $\Cent_G(H)$ denote the centralizer of $H$ inside $G$ and $\Cent_G(H)^0$ its identity 
component.
 
  Then there exists a closed subset $Y$ of $G$ such that
 \begin{enumerate}
  \item[1)] 
  on the one hand we have 
  $$G=Y\cdot \Cent_G(H)^0,$$
  \item[2)] 
  on the other hand, given 
    \renewcommand{\labelenumii}{(\roman{enumii})}
  \begin{enumerate}
  \item a subset $\Omega$ of $H$ with nonempty interior,
  \item a finite dimensional linear representation $\rho:G\to \GL(V)$,
  \item and a norm $\Nm{-}$ on $V$,
  \end{enumerate}
   there exists a constant $c>0$ such that
   \begin{equation}\label{ineqlemma}
    \forall y\in Y,\forall v \in V, \sup_{\omega\in\Omega} 
    \Nm{\rho(y \cdot \omega)(v)} \geq c\cdot\Nm{v}.
   \end{equation}
 \end{enumerate}
\end{theorem}

\renewcommand{\labelenumi}{\arabic{enumi}.} 

This result generalizes \cite[Proposition~4.4]{EMSGAFA}, where $H$ is assumed to be an 
algebraic torus. The Theorem~\ref{lemma} will be formulated for ``reductive Lie groups'' $G$ and 
$H$ in \S\ref{general}. 

The next result complements the above theorem to provide a more complete picture. 

\begin{theorem}[Linearized Focusing] \label{theofocusing}
We consider the setup of Theorem~\ref{lemma}. Let $Y$ be a subset of $G$ satisfying the 
conclusion~2) of Theorem~\ref{lemma}.

\begin{enumerate}
\item[3)] Then given
\begin{enumerate}\renewcommand{\labelenumii}{(\roman{enumii})}
  \item a sequence $(y_n)_{n\in\N}$ in $Y$,
  \item a bounded subset $\Omega$ of $H$ with nonempty interior,
  \item a finite dimensional linear representation $\rho:G\to \GL(V)$,
  \item a vector $v$ of $V$,
\end{enumerate}
we have the equivalence between the following properties
\begin{enumerate}\renewcommand{\labelenumii}{(\Alph{enumii})}
\item the sequence $(y_n \Omega v)_{n\in\N}$ is uniformly bounded in $V$, \label{Vborne}
\item the sequence $(y_n)_{n\in\N}$ is  bounded in $G$ modulo the point-wise stabilizer of 
$\Omega v$. \label{Gborne}
\end{enumerate}
\end{enumerate}
\end{theorem}

\subsubsection{About Theorem~\ref{lemma}} \label{subsec:lemma}
 Roughly Theorem~\ref{lemma} means that one cannot uniformly contract a piece 
 $\Omega\cdot v$ of a $H$-orbit if one acts with an element $y$ which is 
 ``orthogonal''\footnote{Cf.\ Theorem~\ref{Mostow} below.} to the centralizer of $H$. 
 The heuristic is the following: might $v$ itself be contracted by $y$, the $y\Omega y^{-1}$ part in 
 \begin{equation}\label{omegaexp}
 y\cdot\Omega v=y\Omega y^{-1}\cdot yv  
 \end{equation}
 would be sufficiently expanded in some direction. For $y$ in $\Cent_G(H)$, (\ref{omegaexp}) 
yields $y\cdot\Omega v=\Omega\cdot yv$. Assuming that $\Omega$ is bounded, the inequality 
in~(\ref{ineqlemma}) cannot hold with a uniform constant for $y$ in $\Cent_{G}(H)$ and $v$ in 
$V\smallsetminus\{0\}$, provided that
 ${\Nm{yv}}/{\Nm{v}}$ can take arbitrarily small values (equivalently, $\rho(\Cent_{G}(H))$ is not 
compact). As a result, the first condition of Theorem~\ref{lemma} is essentially optimal, for such 
an $\Omega$. 
 
 Our proof uses a different and novel approach as compared to the original proof of 
\cite[Proposition~4.4]{EMSGAFA}, and hence it is applicable in greater generality. It also involves 
a general decomposition theorem of Mostow (Theorem~\ref{Mostow} below), compatibility of 
Cartan decompositions under Lie algebra representations (Remark~\ref{rem:remarks}), and the 
convexity of the exponential function (Proposition~\ref{cvx}).

Actually, we will prove Theorem~\ref{lemma} under weaker hypotheses on $\Omega$ 
(condition~$(*)$ of Corollary~\ref{constant}), for an explicitly defined subset $Y$ 
(see~\ref{notations}), and 
obtain an effective constant $c$ (see equation~(\ref{effectiveconstant})). 

\subsubsection{About~Theorem~\ref{theofocusing}} \label{subsec:focusing}
Let us spell out the terminology used. In property~(A), the sequence of subsets 
$(y_n \Omega v)_{n\in\N}$ is said to be \emph{uniformly bounded} if there is a compact 
subset of $V$ 
containing simultaneously each  of the $y_n \Omega v$; equivalently, 
\begin{equation}\label{uniformbddness}
\text{for any norm on } V, \exists c\in\R, \forall n\in \N, 
\forall \omega\in\Omega, \Nm{y_n \omega v}\infa c.
\end{equation}
Concerning the property $(B)$,  the \emph{point-wise stabilizer} of $\Omega v$ denotes the 
subgroup 
\begin{equation} \label{stabilizer}
F=\bigl\{ g\in G~\bigl|~\forall \omega \in \Omega, g \omega v=\omega v \bigr\}
\end{equation}
of elements of $G$ fixing point-wise each element of $\Omega v$.

The following are notable differences with~Theorem~\ref{lemma}. We are interested in 
boundedness for the  translated piece of orbit $y_n \Omega v$, instead of uniform convergence 
to $0$; at the very end of the proof of Theorem~\ref{theofocusing} we deduces the uniform 
convergence from Theorem~\ref{lemma}. The subset $\Omega$ is therefore required to be 
bounded (and here again, one can replace the nonempty interior condition by Zariski density if 
$H$~is linear, and even weaker conditions, allowing $\Omega$ to be finite are possible). More 
importantly, one fixes a vector $v$ of $V$. The  focusing condition (B) depends greatly on $v$.

\begin{remark} The $p$-adic analogue of Theorem~\ref{lemma} has been established by 
Richard~\cite[\emph{Expos\'{e}~V}]{These} by following up and extending some of the ideas of 
the current article through extensive work. Its $S$-adic analogue is deduced 
in~\cite[\emph{Expos\'{e}~VI}]{These}. The $p$-adic analogue of Theorem~\ref{theofocusing} 
follows from Remark~\ref{rem:kLC}, and the corresponding $S$-arithmetic analogue easily 
follows and is derived and used in~\cite{focusing}.
\end{remark}
 
\subsection{Non-divergence of translated homogeneous measures}\label{sec:nondiv}

The following main theorem of \cite{EMSGAFA} can be obtained from the general set up of 
\cite{KM-Annals} and Theorem~\ref{lemma} in a straightforward manner as explained below 
(cf.~Corollary~\ref{cor:invariant}). 

\begin{theorem}[Eskin-Mozes-Shah {\cite[Theorem~1.1]{EMSGAFA}}] \label{thm:gafa}
Let $G$ be reductive real algebraic group defined over $\Q$ with no nontrivial $\Q$-characters, 
and $\Gamma\subset G(\Q)$ be a lattice in $G$. Let $H$ be a reductive subgroup defined over 
$\Q$ such that $Z(H)$ is $\Q$-anisotropic. Let $\mu_{H}$ denote the $H$-invariant probability 
measure on $H/H\cap\Gamma$. Then the collection of measures $\{g\mu_{H}:g\in G\}$ is 
pre-compact in the space of probability measures on $G/\Gamma$.
\end{theorem}

In \cite[Proposition~4.4]{EMSGAFA} the Theorem~\ref{lemma} was proved for abelian $H$. As 
Theorem~\ref{lemma} was not available for semisimple groups $H$, Theorem~\ref{thm:gafa} 
was first proved for semi-simple groups without compact factors using non-divergence of 
unipotent flows due to Dani and Margulis~\cite{DM-MathAnn}, and then for compact simple 
factors in a very indirect manner using a trick involving its complexification. Then another 
argument was required for combining the toral and semisimple parts of $H$. This method is 
somewhat artificial, and hence does not generalize in a natural way. Therefore 
Theorem~\ref{lemma} is an important missing component from \cite{EMSGAFA} 
from the point of view of 
having a natural method of proving non-divergence of translates of pieces orbits of reductive 
groups on homogeneous spaces. 

It might be worthwhile to note that our proof is effective in most aspects. In particular, it makes it 
possible to quantitatively estimate the divergence of translated measures in the context 
of~\cite{EMSGAFA}. 

\begin{remark}
We note that the $S$-arithmetic generalization of Theorem~\ref{thm:gafa} has been obtained in 
\cite[Exp VI]{These}, by putting together \cite{KT}, Theorem~\ref{lemma} and its $p$-adic 
analogue in \cite[Exp V]{These} in a straightforward manner. This more natural approach  allows 
some improvements on Theorem~\ref{thm:gafa} above: the real Lie group $H$ need not be 
defined over $\Q$, not even algebraic over $\R$, and only assume that $Z(H)$ projects to a 
pre-compact set in $G/\Gamma$. 
\end{remark}

\subsection{Equidistribution of translates and the focusing condition} 

One of our goals is to analyze the limit distribution of $y_{n}\mu_{H}$ for a sequence 
$y_{n}\to \infty$ in $Y$ as in Theorem~\ref{theofocusing}. For this purpose we consider a small 
pre-compact nonempty open set $\Omega\subset H$ with its boundary having zero Haar measure of 
$H$, and consider the Haar measure of $H$ restricted to $\Omega$, normalize it and let 
$\mu_\Omega$ be its pushforward on $G/\Gamma$.  We want to understand the weak$^\ast$ 
accumulation points of the sequence of translated measures $\{y_n\mu_\Omega\}_{n\in\N}$. In 
view of non-divergence of unipotent trajectories after Margulis~\cite{Mar-nondiv} and 
Dani~\cite{Dani-nondiv,Dani-rk2}, and its generalisations in 
\cite{DM-asymptotic,EMSGAFA,Shah-horo,KM-Annals,KT} one reduces to showing that the sequence of 
measures $\{y_n\mu_\Omega\}$ is pre-compact in the space of probability measures on $G/\Gamma$ if 
and only if for certain 
finite dimensional representation $V$ of $G$ and any nonzero vector $p\in V$ with $\Gamma p$ 
discrete, the set  $y_n\Omega \Gamma p$ avoids a fixed neighborhood of $0$ for all $n$.  This 
condition is precisely the conclusion of Theorem~\ref{lemma} as explained in 
\ref{subsec:lemma}. Therefore we pass to a subsequence of $\{y_n\}$ so that the sequence 
$\{y_n\mu_{\Omega}\}$ converges to a probability measure $\mu$ on $G/\Gamma$. Again using 
Theorem~\ref{lemma} and Theorem~\ref{theofocusing} for a specific representation, one can 
show that $\mu$ is invariant under a nontrivial unipotent subgroups of $G$. More precisely, one 
obtains

\begin{corollary} \label{cor:invariant} Let the notation be as in Theorem~\ref{theofocusing}. 
Suppose further that $\{y_{n}\}$ has no convergent subsequence and $\partial\Omega$ admits 
null Haar measure of $H$, and consider the restriction of the Haar measure of $H$ to $\Omega$. 
Let $\mu_{\Omega}$ denote its normalized pushforward on the image of $\Omega$ on $G/
\Gamma$, where $\Gamma$ is a lattice in $G$. Then after passing to a subsequence, 
$y_{n}\mu_{\Omega}$ converges to a probability measure $\mu$ on $G/\Gamma$, and $\mu$ is 
invariant under a nontrivial $\Ad$-unipotent one-parameter subgroup of $G$.
\end{corollary}

Then we can apply Ratner's classification~\cite{R-measure} of ergodic invariant measures for the 
unipotent group action to $\mu$. Using the linearization techniques as developed in 
\cite{DS,DM-MathAnn,Shah-uniform,DM-Limit,EMS,Tom-arithmetic}, we can show that if for certain finite 
dimensional representation $V$ and a discrete orbit  of a nonzero $p\in V$ under $\Gamma$, the 
sets $y_n \Omega\Gamma p$ avoid any given large ball in $V$ for all large $n$, then $\mu$ is 
invariant under $G$; that is, the sequence $\{y_n\mu_\Omega\}$ is equidistributed in $G$. 
Therefore if $\mu$ is not $G$-invariant then there exists a sequence 
$\{\gamma_n\}\subset\Gamma$ such that $\{y_n\Omega \gamma_n p\}$ is uniformly bounded. 
Because of 
Theorem~\ref{lemma}, $\{\gamma_{n}p\}$ must be bounded, and since $\Gamma p$ is discrete, by 
passing to a subsequence, we have that $\gamma_{n}p=\gamma_{1} p:=v$ for all $n$. Now we 
can apply Theorem~\ref{theofocusing} to deduce that $\{y_{n}\}$ is bounded modulo the 
point-wise stabilizer of $\Omega v$. Since $\Omega$ is Zariski dense in $H$, we have that 
$\{y_{n}\}$ is bounded in $G$ modulo $\bigcap_{h\in H} h G_{v} h^{-1}$, where $G_{v}$  being the 
stabilizer of $v$ in $G$. This is a linear algebraic condition relating $\{y_n\}$, $H$, and 
$G_{\gamma_{1}p}$. This condition is referred to by  the {\em focusing\/} condition. The focusing 
condition is further analyzed by Richard and Zamojski~\cite{focusing} to obtain very general 
results on limiting distributions of translates of measures.

\subsubsection{On applications in Arithmetic Geometry}
Various new arithmetical applications of Theorems~\ref{lemma} and \ref{theofocusing} are 
obtained in \cite{focusing} by extending the above strategy to the $S$-arithmetic setting. This 
method would also strengthen some of the results proved in \cite{GO:Rational}. 

One type of application of this work is to counting problems, in a setup analogous to 
Manin-Peyre conjecture for homogeneous varieties $G/H$ (cf.\ \cite{GMO} for the case of symmetric 
subgroups $H$). These problems can be translated to a homogeneous dynamical problem about 
translates of measures through the unfolding argument. The dynamical problem can be divided 
into sub-issues: non-divergence (see~\cite[\emph{Expos\'{e}~VI}]{These}); 
focusing (see~\cite{focusing}); and volume computations (\cite{EMS,ACLT,GO:Rational}). 

Another type of application of the measure classification, via adelic mixing or equidistribution of 
Hecke points, was, in the thesis~\cite{These} providing a conditional proof of an  
refinements of a conjecture of Pink: equidistribution of sequences of Galois orbits originating 
from a given Hodge-generic Hecke orbit. This was conditioned to a form of the Mumford-Tate 
conjecture.

The preprint~\cite{focusing} is expected to allow an unconditional statement, though restricted to 
$S$-arithmetic Hecke orbits. Both the above applications involve Theorems~\ref{lemma} and 
\ref{theofocusing} and their $p$-adic analogues.

\subsubsection*{Acknowledgement}
The authors would like to acknowledge the support of the following institutions where the 
collaboration on this work took place: Tata Institute of Fundamental Research (Mumbai, India), 
Universit\'{e} Paris-Sud 11 (Orsay), \'{E}cole polytechnique f\'{e}d\'{e}rale de Lausanne,  ETH Z
\"{u}rich and The Ohio State University (Columbus). The first named author (R.Richard) would 
also like to acknowledge the support of  \'{E}cole normale sup\'{e}rieure de Paris and Universit
\'{e} de Rennes 1. We thank Donna Testerman for pointing out the reference to Kempf's 
paper~\cite{Kempf}. We would like to thank the anonymous referee for careful review of this article, and the
suggestions to help improve its readability. 

\section{Preliminaries}\label{preliminaires}
\emph{We assume all Lie algebras to be real or complex, and finite dimensional.}

Let us first recall some more or less well known facts on Cartan involutions and Cartan 
decompositions. This is for convenience and because no precise references were found for 
some of these facts; most importantly the criterion~(c) of \S\ref{conventions} and functoriality 
properties of the Killing form \ref{functoriality}. In order to prove this criterion~(c), we use a 
variant of \cite[Theorem~6]{Mostow} in the general context of fully reducible linear Lie groups as 
in~\cite[Theorem~7.3]{Mostowherm}. This variant, although not stated, actually follows from 
proofs of~\cite{Mostowherm}. Later will provide more precise definitions, as we do not assume 
our Lie groups to be \emph{linear}; i.e.\ to admit a finite dimensional faithful representation. Our 
proof of Theorem~\ref{lemma} will only rely on criterion~(c) of \S\ref{conventions} (proved 
in~\ref{proofc}), on construction~\ref{constructionnorm}, and on facts collected in 
Remark~\ref{rem:remarks}.

\subsection{Cartan involutions}\label{Cartan}
Let $\lie{g}$ be a finite dimensional real Lie algebra and denote its adjoint representation by 
$\ad:\lie{g}\to\lie{gl}(\lie{g})$. Recall that the \emph{Killing form} on $\lie{g}$ denotes the real 
bilinear form on $\lie{g}\times\lie{g}$ which sends $(X,Y)$ to 
$B(X,Y)=\Tr(\ad(X)\ad(Y))$ \cite[\S3~N$^{0}$6~Def.~4]{B}. This form is obviously symmetric, and is nondegenerate if and only if $\lie{g}$ is semisimple \cite[\S6 N$^{0}$1~Theorem~1]{B}.

\subsubsection{}\label{Killinginvariant} If $G$ is a \emph{linear}\/ real Lie group with Lie algebra 
$\lie{g}$, then isotropic nonzero vectors of $B$ are exactly the generators of (one dimensional and 
non trivial) $\Ad$-unipotent subgroups; nonzero non-isotropic vectors of positive (resp.\ negative) 
norm are exactly the generators of noncompact one-parameter subgroups of semisimple 
elements (resp.\ one-parameter compact subgroup). The Killing form is \emph{completely 
invariant\/} \cite[\S3 N$^{0}$6~Proposition~10]{B}; equivalently, the image $\exp(\ad(\lie{g}))$ of 
$\ad(\lie{g})$ in $GL(\lie{g})$ is contained in the orthogonal group of $B$. In particular, for any $X$ 
and $Y$ in $\lie{g}$, $\exp(\ad(X))(Y)$ is isotropic (resp.\ positive, negative) if and only if $Y$ is.

\subsubsection{}\label{defCartan}A \emph{Cartan involution} of a real Lie algebra $\lie{g}$ means 
an involution $\theta$ of the algebra $\lie{g}$ such that the bilinear form 
$(X,Y)\mapsto B_\theta(X,Y)=-B(X,\theta(Y))$ is symmetric and strictly positive 
definite~\cite[III 7]{Helgason}. In particular $B$ is nondegenerate and $\lie{g}$ is semisimple. 

\subsubsection{}\label{orthogonals} Note that if a linear subspace $\lie{z}$ of $\lie{g}$ is invariant 
under $\theta$, then its orthogonal complements with respect to $B$ and $B_\theta$ coincide. 
Consequently, as $B_\theta$ is anisotropic, $\lie{z}$ and its orthogonal complement are 
supplementary. Moreover, $\theta$ stable subspace are stable under taking adjoint with respect 
to $B_\theta$.

\subsubsection{}\label{compatibility}Consider the adjoint representation 
$\ad:\lie{g}\to\lie{gl}(\lie{g})$ and a Cartan involution $\theta$ on $\lie{g}$. 
Then for any $X\in \lie{g}$, the negative of the adjoint with respect to $B_{\theta}$ of the 
endomorphism $\ad_{X}$ of $\lie{g}$ equals $\ad_{\theta(X)}$. 

\begin{proof}
 We need to show that for any $Y$ and $Z$ in $\lie{g}$,
\begin{equation*}
B_\theta(\ad_X(Y),Z)=B_\theta(Y,-\ad_{\theta(X)}Z).
\end{equation*}
By definition of $\ad$ and $B_\theta$ this equality means
$$B(-[X,Y],\theta(Z))=B(Y,\theta([\theta(X),Z])).$$
Now $B([X,Y],\theta(Z))+B(Y,[X,\theta(Z)])=0$ follows from invariance of the Killing form, whereas 
the identity $\theta([\theta(X),Z])=[X,\theta(Z)]$ follows from the fact that $\theta$ is an algebra 
involution.
\end{proof}

\subsubsection{}\label{GlobalCartan}
A \emph{global Cartan involution} of a connected real Lie group $G$ with Lie algebra $\lie{g}$ is 
an involution $\Theta:G\to G$ whose differential at the neutral element is a Cartan involution of 
$\lie{g}$. Every Cartan involution of $\lie{g}$ extends to $G$.
\begin{proof}
 We may assume $G$ is semisimple, else there is no Cartan involution. According 
 to~\cite[\S6~N$^{0}$2~Theorem~1]{B3}, any Cartan involution $\theta$ has an extension 
 $\widetilde{\Theta}$ to 
the universal cover $\widetilde{G}$ of $G$. It will be enough to show that $\widetilde{\Theta}$ 
fixes element-wise the center $Z\bigl(\widetilde{G}\bigr)$ of $\widetilde{G}$. As $Z
\bigl(\widetilde{G}\bigr)$ is a characteristic subgroup, it is stable under automorphisms. Hence 
$\widetilde{\Theta}$ descends to an involution ${\Theta}^\ad$ of the adjoint group 
$G^\ad=\widetilde{G}\sur Z\bigl(\widetilde{G}\bigr)$, which is linear. It will be enough to show 
that the induced action of $\widetilde{\Theta}$ on the fundamental group $\pi_1(G^\ad)$ is trivial. 
But, from Cartan decomposition~\cite[Theorem~3.2]{Mostowherm}, $G^\ad$ retracts to any of its 
maximal compact subgroups, and one such a maximal compact subgroup is the fixed locus of 
$\Theta^\ad$.
\end{proof}
 
\subsubsection{}\label{projcomp}We will say that a closed subgroup of a \emph{connected 
semisimple} Lie group $G$ is \emph{projectively compact} if its image under the adjoint 
representation $\Ad:G\to GL(\lie{g})$ is compact. When $G$ is \emph{linear}, this property is 
equivalent to compactness. In general this property is stable under direct image by morphisms of 
semisimple connected Lie groups, and inverse images by isogenies. When $G$ is connected, 
semisimple and \emph{linear}, the set of fixed point of a given global Cartan involution defines a 
maximal compact subgroup (recall that $G$ is connected). If $G$ is only assumed to be 
connected and semisimple, writing $G$ as a central covering of its adjoint group, we deduce that 
the set of fixed point of a given global Cartan involution defines a maximal projectively compact 
subgroup.

\subsubsection{}\label{compactalgebra} A semisimple real Lie algebra is said to be~
\emph{compact\/} if its Killing form is totally negative. One actually needs only to ask the Killing 
form to be \emph{anisotropic}.

\subsubsection{} Every semisimple Lie algebra $\lie{g}$ is a direct product of its simple ideals 
\cite[\S6~N$^{0}$2~Corollary~1]{B}, and these simple ideals are pairwise orthogonal for the Killing 
form \cite[\S6 N$^{0}$1~Corollary~1]{B}. For any Cartan involution $\theta$ of $\lie{g}$, one has 
$-B(X,\theta(X))>0$ whenever $X\neq0$. Consequently the image $\theta(X)$ of a non zero 
element $X$ of $\lie{g}$ cannot be orthogonal to $X$ with respect to the Killing form. 
\emph{A fortiori} $\theta$ cannot send any simple ideal of $\lie{g}$ to an orthogonal ideal. 
Consequently, a 
Cartan involution of $\lie{g}$ stabilizes each simple ideal of $\lie{g}$. As ideals of $\lie{g}$ are 
sum of simple ideals \cite[\S6~N$^{0}$2~Corollary~1]{B}, a Cartan involution of $\lie{g}$ stabilizes 
each ideal of $\lie{g}$.

\subsubsection{} According to \cite[\S3~N$^{0}$6~Proposition~9]{B}, the restriction of the Killing form 
of $\lie{g}$ to an ideal $\lie{a}$ is the Killing form of $\lie{a}$. It follows that the restriction of a 
Cartan involution to an ideal is a Cartan involution, and that an endomorphism of $\lie{g}$ is a 
Cartan involution if and only if it stabilizes each simple ideal and its restriction to each simple 
ideal is a Cartan involution. These restrictions being nondegenerate, a simple ideal does not 
intersect its orthogonal: the orthogonal of a simple ideal, and more generally of any ideal 
$\lie{a}$, is the sum of simple ideals not contained in $\lie{a}$. 

\subsubsection{}\label{functoriality}  Consider a semisimple Lie algebra $\lie{g}$ and a morphism 
of Lie algebras $\phi:\lie{g}\to\lie{g}'$. Its kernel $\ker(\phi)$ is an ideal whose orthogonal 
$\ker(\phi)^\bot$ is a supplementary ideal; $\phi$ sends $\ker(\phi)^\bot$ bijectively onto 
$\phi(\lie{g})$. In particular $\phi(\lie{g})$ is a semisimple Lie algebra; any Cartan involution of 
$\lie{g}$ stabilizes $\ker(\phi)$; the induced involution on $\phi(\lie{g})$ is a Cartan involution. We 
will call the latter the \emph{image Cartan involution}.

\subsubsection{}\label{decomposition} Given a Cartan involution $\theta$ on a semisimple Lie 
algebra $\lie{g}$, the associated \emph{Cartan decomposition} denotes the decomposition 
$\lie{g}=\lie{k}\oplus\lie{p}$ of $\lie{g}$ as a direct sum of the eigenspaces $\lie{k}$ (resp.\ $\lie{p}$) 
of $\theta$ associated with the eigenvalue $+1$ (resp.$-1$). Clearly Cartan involutions and 
associated Cartan decompositions determine each other. By definition $\theta$ is self-adjoint 
with respect to $B_{\theta}$, hence the eigenspaces $\lie{k}$ and $\lie{p}$ are orthogonal 
complements of each other. Consequently, the Cartan involution $\theta$ is determined by 
$\lie{k}$ only (knowing $B$). Note that $\lie{k}$ is a maximal negative anisotropic subspace of 
$B$ and a subalgebra of $\lie{g}$ (it satisfies Frobenius integrability condition), and that $\lie{p}$ is 
maximal positive anisotropic linear subspace and is $\ad(\lie{k})$-invariant.

\subsubsection{}\label{Cartancomplex} Let $\lie{g}_\C=\lie{g}\tens_\R\C$, the complexified Lie 
algebra obtained from $\lie{g}$ by extending the base field. 
According to \cite[III.~Proposition~7.4]{Helgason}, every \emph{Cartan decomposition} of $\lie{g}$, as 
defined in \cite[p.~183]{Helgason} actually comes from a Cartan involution, as defined above, 
and these Cartan decompositions are exactly the restriction to $\lie{g}$ of the Cartan 
decompositions of the real Lie algebra $\lie{g}_\C$ which are invariant (factor-wise) under 
complex conjugation. The anisotropic subalgebra of $\lie{g}_\C$ corresponding to the latter are 
the~\emph{compact} (see \ref{compactalgebra}) real forms of the complex Lie algebra 
$\lie{g}_\C$ which are \emph{invariant} under the complex conjugation relative to the real 
structure induced by $\lie{g}$ (see also~\cite[section~2]{Mostowherm}).

\subsubsection{}\label{invariantcompactrealforms} For a \emph{reductive} lie algebra $\lie{g}$, 
\emph{together with} an embedding $\lie{g}\to\lie{gl}(V)$, for some $V$ of finite dimension, a 
\emph{real form} $\lie{k}$ of $\lie{g}_\C$ is said to be \emph{compact} if $\exp(\lie{k})$ is 
compact in $\GL(V\tens\C)$, and \emph{invariant} if $\lie{k}$ is invariant under the complex 
conjugation on $\lie{g}_\C$ relative to $\lie{g}$. According to~\cite[Lemma~6.2]{Mostowherm}, 
there exists such a compact form if and only if the action of $\lie{g}$ on $V$ is semisimple (the 
``only if'' part is known as ``Weyl's unitary trick''). In such a case, we will say that $\lie{g}$ 
together with its embedding is a \emph{linear fully reducible\/} subalgebra of $\lie{gl}(V)$.

\subsubsection{}\label{nestedextension} The preceding points imply that invariant compact forms 
generalizes to linear fully reducible Lie algebra the Cartan decompositions of semisimple Lie 
algebras. By using~\cite[Theorem~4.1]{Mostowherm}  as in proof of \cite[Theorem~5.1$^{1}$]
{Mostowherm}, given a nested sequence of linear fully reducible subalgebras of $\lie{gl}(V)$, one 
can form a nested sequence of invariant compact real forms of the corresponding complexified 
linear algebras.

\subsubsection{}\label{constructionnorm} Consider now a finite dimensional linear representation 
$\rho:\lie{g}\to\lie{gl}(V)$ of a semisimple Lie algebra. Then, given a Cartan involution of $\lie{g}$ 
we get a Cartan involution of $\rho(\lie{g})$, by~\ref{functoriality}. This Cartan involution 
corresponds to an invariant compact real form on $\rho(\lie{g})_\C$, by~\ref{Cartancomplex}, 
which we can extend to $\lie{gl}(V)$, by~\ref{nestedextension}. 

Any such extension is the unitary group of a euclidean structure on $V$, unique up to 
proportionality. Now $\rho(\lie{g})$ is stable under taking adjoint with respect to the euclidean 
structure on $V$~\cite[proof of Theorem~5.1]{Mostowherm}, and such that the euclidean 
adjunction extends the negative of the image Cartan involution~\ref{functoriality} on $\rho(\lie{g})
$~(\emph{loc. cit.}, equation~(2)).

\subsubsection{}\label{proofc} Criterion~(c) of \S\ref{conventions} clearly follows from the 
corresponding statement at the level of Lie algebras, which we now prove. Namely \emph{a real 
subalgebra $\lie{h}$ of a (finite dimensional) real semisimple Lie algebra $\lie{g}$ is invariant 
under some Cartan involution of $\lie{g}$ if and only if the adjoint action of $\lie{h}$ on $\lie{g}$ is 
fully reducible.}
\begin{proof}
 Consider a subalgebra $\lie{h}$ of a semisimple Lie algebra $\lie{g}$. If $\lie{h}$ is invariant 
under $\theta$, its image under $\ad:\lie{g}\to\lie{gl}(\lie{g})$ is stable under taking adjoint with 
respect to $B_\theta$, according to~\ref{compatibility}. As $B_\theta$ is \emph{anisotropic}, the 
orthogonal complement of a $\ad_{\lie{h}}$-stable subspace defines a \emph{supplementary} 
$\ad_{\lie{h}}$-stable subspace. It implies that $\lie{h}$ acts fully reducibly on $\lie{g}$.

 Assume now that $\lie{h}$ acts fully reducibly on $\lie{g}$. We shall denote  the extension of 
scalars by subscripts. Then the linear subalgebra $\ad(\lie{h})_\C$ of $\lie{gl}(\lie{g})_\C$ has an 
invariant compact real form~\ref{invariantcompactrealforms}. This real form is contained in an 
invariant compact real form of $\ad(\lie{g})_\C$~\ref{nestedextension}. The latter is associated 
with a Cartan involution $\theta$ of $\lie{g}$~\ref{Cartancomplex}. Applying to $\lie{h}$ the Lie 
algebra analogue of decomposition~(2) in proof of Theorem~5.1 in~\cite{Mostowherm}, we see 
that $\lie{h}$ is invariant under $\theta$, each factor being contained in a factor of the 
corresponding Cartan decomposition of $\lie{g}$.
\end{proof}

\subsection{Notational conventions}\label{notations} In the next sections we will often consider 
the following situation. Let us fix, once for all, our notations.
\subsubsection{General notations} \label{generalnotations}
Let $G$ be a connected semisimple Lie group, let $H$ be a connected reductive Lie subgroup in 
$G$. According to criterion~(c) of \S\ref{conventions}, let $\Theta$ be a global Cartan involution 
(\cf~\ref{defCartan}, \ref{GlobalCartan}) of $G$ under which $H$ is invariant. We denote by 
$\Cent_{G}(H)$ the centralizer of $H$ in $G$, and by $K$ the maximal projectively compact 
subgroup consisting of fixed points of $\Theta$ (\cf~\ref{projcomp}).

Denote by $\lie{g}$, $\lie{h}$, $\lie{z_g}$, $\lie{k}$ the Lie algebra of $G$, $H$, $\Cent_{G}(H)$, 
and $K$ respectively, and denote by $\theta$ the differential of $\Theta$ at the identity element.

We write $\lie{k}^\bot$ and $\lie{z_g}^\bot$ for the orthogonal complements, with respect to the 
Killing form (\cf~\ref{Cartan}) on $\lie{g}$, to $\lie{k}$ and $\lie{z_g}$ respectively. We define 
$\lie{p}=\lie{k}^\bot\cap\lie{z_g}^\bot$, $P=\exp_{G}(\lie{p})$~and $Y=K\cdot P$. Finally, 
$B_\theta:\lie{g}\times\lie{g}\to\R$ will be the strictly positive definite symmetric bilinear form on 
$\lie{g}$ associated with $\theta$ (\cf~\ref{defCartan}).

\subsubsection{Relative notations} \label{relativenotations} When considering, in situation 
\ref{generalnotations}, a finite dimensional linear representation $\rho:G\to GL(V)$, we shall use 
the following notations. 

We denote by $\d\rho:\lie{g}\to\lie{gl}(V)$ the differential of the representation $\rho$ and by 
$\lie{z}$ the centralizer of $\rho(H)$ in $\lie{gl}(V)$. 
The Trace map on $\lie{gl}(V)$ is denoted by $\Tr:\lie{gl}(V)\to\R$, and the \emph{trace form\/} 
means the bilinear form $(X,Y)\mapsto\Tr(XY)$ on $\lie{gl}(V)$ (it is the \emph{bilinear form 
associated with the $\lie{gl}(V)$-module $V$}, according to \cite[\S3~N$^{0}$6~Definition~4]{B}).
We write $\lie{z}^\bot$ for the orthogonal complement of $\lie{z}$ with respect to  the trace form.

\subsubsection{Choice of a euclidean structure} \label{subsec:euclidean} Using 
\ref{constructionnorm}, we can fix a euclidean structure (inner product) on $V$ such that the 
involution $\theta_V:\lie{gl}(V)\to\lie{gl}(V)$ given by $X\mapsto -X^{\mathrm{T}}$, the negative of 
the euclidean adjoint, stabilizes $\d\rho(\lie{g})$ and extend the image Cartan involution of 
$\theta$ on $\d\rho(\lie{g})$.

\begin{remark} \label{rem:remarks}
In such a situation, using our observations in section~\ref{Cartan} on Cartan involutions we will 
deduce the following:

\begin{enumerate}
\item In $\lie{g}$, the subspaces $\lie{k}$, $\lie{h}$, and hence $\lie{z_g}$, are  invariant under 
$\theta$.
\item The orthogonal complements of $\lie{k}$ (resp.\ $\lie{h}$, $\lie{z_g}$) in $\lie{g}$ with 
respect to the Killing form or with respect to $B_\theta$ are the same. This orthogonal 
complement is supplementary to $\lie{k}$ (resp.\ $\lie{h}$, $\lie{z_g}$) in $\lie{g}$.
\item \label{remarks:3}The subspace $\lie{z_g}^\bot$ of $\lie{g}$ is invariant under the adjoint 
action of $H$; it is the unique supplementary $H$-stable subspace to the isotypic component 
$\lie{z_g}$ in the $H$-module $\lie{g}$.
\item The map $\d\rho:\lie{g}\to\lie{gl}(V)$ commute with the involutions $\theta$ on $\lie{g}$ and 
$\theta_V$ on $\lie{gl}(V)$.
\item The subspace $\lie{z}^\bot$ of $\lie{gl}(V)$ is invariant under the adjoint action of $H$; it is 
the unique supplementary $H$-stable subspace to the isotypic component $\lie{z}$ in the $H$-
module $\lie{gl}(V)$.
\item\label{rkzbar} The map $\d\rho:\lie{g}\to\lie{gl}(V)$ sends $\lie{z_g}$ to $\lie{z}$ and 
$\lie{z_g}^\bot$ to $\lie{z}^\bot$.
\item In $\lie{gl}(V)$, the subspaces $\d\rho(\lie{k})$, $\d\rho(\lie{z_g})$, $\d\rho(\lie{h})$ and 
hence $\lie{z}$ and $\lie{z}^\bot$, are  invariant under $\theta_V$.
\item\label{rkpiz} The orthogonal projection $\pi_\lie{z}:\lie{gl}(V)\to\lie{z}$ with respect to the 
trace form commutes with $\theta_V$; it sends self-adjoint endomorphisms to self-adjoint 
endomorphisms.
\item\label{rkself} The map $\d\rho$ sends elements of $\lie{k}^\bot$ to self-adjoint 
endomorphisms of $V$.
\end{enumerate}
\end{remark}

\begin{proof}
 \textbf{1} follows from the definition of $\lie{k}$, the assumption on $H$, and the construction of 
$\lie{z_g}$ from $\lie{h}$. \textbf{2} follows from~1, definition of $B_\theta$ and that $B_\theta$ is 
anisotropic (\cf~\ref{orthogonals}).
To prove  \textbf{3}, note that both $\lie{z_g}$ and $B$ are invariant under adjoint action of $H$; 
and that the isotypic components are uniquely defined. \textbf{4} follows from the choice of 
$\theta_V$ in~\ref{subsec:euclidean}. \textbf{5} follows from the same arguments as in~3. To 
prove \textbf{6}, note that as $\d\rho$ commutes with the adjoint action of $H$, it preserves the 
isotypic decomposition. \textbf{7} follows from~1, from~4, from the construction of $\lie{z}$ from 
$\d\rho(\lie{h})$, and that $\Theta_{V}$ preserves the trace form. 
\textbf{8} follows from the observations that both the subspace $\lie{z}$ and the orthogonality 
condition are invariant under $\theta_V$, and that self-adjoint means fixed by $-\theta_V$. 
 \textbf{9} follows from~4. 
 \end{proof}

\section{Effective statements}\label{sectioneffective}
Our effective statements will rely on the Corollary~\ref{coroeffective} of the following result. Note 
that the condition that $p$ ``can only go to infinity in directions orthogonal to $\lie{z_g}$'' is used 
in the next proof only in order to get a uniform lower bound on the eigenvalues.

In this section we follow the notation and convention of section~\ref{notations}. 

\begin{theorem}\label{effective}
 For any $p$ in $P$, the endomorphism $\pi_\lie{z}(\rho(p))$ of $V$ is self-adjoint, positive 
definite, and has no eigenvalue smaller than $1$.
\end{theorem}
\begin{proof}
Fix $p$ in $P$ and write $p=\exp(\wp)$ for some $\wp$ in $\lie{p}$. From 
$\lie{p}=\lie{k}^\bot\cap\lie{z_g}^\bot$ follows that $\d\rho(\wp)$ belongs to both 
$\d\rho(\lie{k}^\bot)$ and $\d\rho(\lie{z_g}^\bot)$. Consequently $\d\rho(\wp)$ is self-adjoint 
(Remark~\ref{rkself}) and 
orthogonal to $\lie{z}$ with respect to the trace form (Remark~\ref{rkzbar}). We will write $S$ for 
$\d\rho(\wp)$. By~\cite[\S4~Corollary~2]{B3}, we get $\rho(p)=\exp(S)$, so that $\rho(p)$ is self-
adjoint and positive definite. As a result, $\lie{\pi_\lie{z}}(\rho(p))$ is self-adjoint 
(Remark~\ref{rkpiz}), and belongs to $\lie{z}$, the image of $\pi_\lie{z}$.

Let $\lambda$ be an eigenvalue of $\pi_\lie{z}(\rho(p))$, and let $\pi_\lambda$ be the 
corresponding spectral projection. We saw that $\pi_\lie{z}(\rho(p))$ is self-adjoint and commutes 
with $H$, and it is well known that $\pi_\lambda$ belongs to the subalgebra generated by 
$\pi_\lie{z}(\rho(p))$. Consequently $\pi_\lambda$ is self-adjoint and commutes with $H$. In 
particular, $\pi_\lambda$ belongs to $\lie{z}$.

 The difference $\pi_\lie{z}(\rho(p))-\rho(p)$ belongs to the kernel of the projection $\pi_\lie{z}$: it 
is orthogonal to $\lie{z}$, and, in particular, to $\pi_\lambda$. Consequently, 
$$\Tr(\pi_\lie{z}(\rho(p))\pi_\lambda)=\Tr(\rho(p)\pi_\lambda)=\Tr(\exp(S)\pi_\lambda).$$ 
 On the other hand, $\Tr(\pi_\lie{z}(\rho(p))\pi_\lambda)$ equals $d_\lambda\cdot\lambda$, 
where $d_\lambda$ is the rank of $\pi_\lambda$ and $d_\lambda>0$.

From Proposition~\ref{cvx} below, we have the inequality 
$$\Tr(\exp(S)\pi_\lambda)\geq d_\lambda\cdot\exp(\Tr(S\pi_\lambda)/d_\lambda).$$ 
Because $\pi_\lambda$ is in $\lie{z}$, and 
$S$ is orthogonal to $\lie{z}$, $\Tr(S\pi_\lambda)=0$. As a consequence $\lambda\geq 1$. Indeed
\begin{equation*}\label{eigenvalue}
 d_\lambda\cdot\lambda=\Tr(\pi_\lie{z}(\rho(p))\pi_\lambda)
 =\Tr(\exp(S)\pi_\lambda) \geq d_\lambda\cdot\exp(\Tr(S\pi_\lambda)/d_\lambda)
 =d_\lambda\cdot 1.
\end{equation*}
\end{proof}
\newcommand{\rank}{\mathbf{rank}}
\begin{proposition}\label{cvx}
 Let $V$ be a finite dimensional euclidean vector space, let $S$ be a self-adjoint endomorphism 
of $V$ and let $\pi$ be a non zero orthogonal projection in $V$. Then follows 
 \begin{equation}\label{convexity}
  \Tr(\exp(S)\pi)\geq\rank(\pi)\cdot\exp\left(\Tr(S\pi)/\rank(\pi)\right).
 \end{equation}
\end{proposition}
\begin{proof}
  Let $S=\sum_\lambda \lambda\cdot\pi_\lambda$ be the spectral decomposition of $S$. Then 
each of the idempotents $\pi_\lambda$ is self-adjoint and 
$\exp(S)=\sum_\lambda \exp(\lambda)\cdot\pi_\lambda$. One computes
\begin{equation}\label{calcul}
\Tr(S\pi)=\sum_\lambda\lambda\cdot\Tr(\pi_\lambda\pi)\text{ and }
\Tr(\exp(S)\pi)=\sum_\lambda\exp(\lambda)\cdot\Tr(\pi_\lambda\pi).
\end{equation}
Since $\pi_\lambda$ and $\pi$ are self-adjoints and idempotents,
\begin{equation}\label{positivity}
\Tr(\pi_\lambda\pi)=\Tr(\pi_\lambda\pi_\lambda\pi\pi)=\Tr(\pi\pi_\lambda\pi_\lambda\pi)
=\Tr((\pi_\lambda\pi)^{\mathrm{T}}\pi_\lambda\pi)\geq 0
\end{equation}
by idempotence of $\pi_\lambda$ and $\pi$, by cyclicity of $\Tr$, by self-adjointness of 
$\pi_\lambda$ and $\pi$, and by positivity of $X\mapsto\Tr(X^{\mathrm{T}}X)$ respectively.

The sum $\sum_\lambda\Tr(\pi_\lambda\pi)$ has value $\Tr(\Id\pi)=\Tr(\pi)=\rank(\pi)$. The 
coefficients $\frac{\Tr(\pi_\lambda\pi)}{\rank(\pi)}$ are well defined, because $\pi$ is assumed to 
be non zero, they are nonnegative, by~(\ref{positivity}), and  they have sum $1$, as 
$\sum_\lambda\Tr(\pi_\lambda\pi)=\rank(\pi)$. From the convexity of the exponential function, one gets
\begin{equation}
\exp\left(\sum_\lambda\lambda\cdot\frac{\Tr(\pi_\lambda\pi)}{\rank(\pi)}\right)\leq
\sum_\lambda\exp(\lambda)\cdot\frac{\Tr(\pi_\lambda\pi)}{\rank(\pi)},
\end{equation}
which, together with (\ref{calcul}), yields inequality~(\ref{convexity}).\end{proof}
\begin{corollary} \label{coroeffective}In situation of Theorem \ref{effective}, for any $p$ in $P$, 
$\pi_\lie{z}(\rho(p))$ is expanding:
\begin{equation}\label{expanding}
\forall v \in V, \forall p\in P, \Nm{\pi_\lie{z}(\rho(p))(v)}\geq \Nm{v}.
\end{equation}
\end{corollary}
\begin{proof} Indeed $\pi_\lie{z}(\rho(p))$ can be diagonalized in an orthonormal basis with all 
diagonal coefficients greater than or equal to $1$.
\end{proof}

\subsection{Application} 
 Consider the adjoint representation $\Ad_\rho$ of $G$ on $\lie{gl}(V)$ by conjugation. Let 
$C(\Ad_\rho)$ be the vector space of functions on $H$ generated by the matrix coefficients of 
$\Ad_\rho$; that is, by functions 
$\mathopen\langle\phi,g{\mathclose\rangle}:h\mapsto\phi(\rho(h)g\rho(h)^{-1})$ for $g$ in 
$\lie{gl}(V)$ and $\phi$ in its algebraic dual $\lie{gl}(V)^\vee$. The 
function $\mathopen\langle\phi,g{\mathclose\rangle}$ depends linearly on both $g$ and $\phi$. 
Consequently $C(\Ad_\rho)$ is finite dimensional: its dimension is bounded by 
$\dim\left(\lie{gl}(V)\tens \lie{gl}(V)^\vee\right)$.

 The following identity has two consequences (cf.~\cite[\S7~N$^{0}$1]{B}):
  \begin{equation}\label{action1}
  \mathopen\langle\phi,g{\mathclose\rangle}(h\cdot h')
  =\mathopen\langle\phi,\rho(h')g\rho(h'^{-1}){\mathclose\rangle}(h).
  \end{equation}
Firstly the space of matrix coefficients of $\Ad_\rho$ is stable under the action of $H$ by 
translation: as a result we get a linear action of $H$ on $C(\Ad_\rho)$. Secondly~(\ref{action1}) 
expresses that, for a fixed $\phi$ in $\lie{gl}(V)^\vee$ the map 
$g\mapsto{\mathopen\langle}\phi,g{\mathclose\rangle}$ is an equivariant morphism from 
$\lie{gl}(V)$ to $C(\Ad_\rho)$.

 Recall that $H$ being \emph{reductive in $G$}, by criterion~(b) of \S\ref{conventions} the 
restriction to $H$ of a finite dimensional representation of $G$ is semisimple. Consequently, in 
both $\lie{gl}(V)$ and $C(\Ad_\rho)$ there is a unique isotypic projection onto the subspaces of 
invariant elements, namely onto the centralizer $\lie{z}$ of $H$ in $\lie{gl}(V)$, and onto the 
subset of constant functions (canonically isomorphic to $\{0\}$ or $\R$ according to $\dim(V)$ 
being zero or not\footnote{If $\dim(V)\neq0$, then $\langle\Tr,\Id\rangle: h\mapsto\dim(V)$ is a 
nonzero constant matrix coefficient.}). Moreover these projections, say $\pi_\lie{z}$ and $\pi_\R$ 
respectively,  commute with any equivariant morphism. In particular, for the morphism $g\mapsto
\langle\phi,g\rangle$ from $\lie{gl}(V)$ to $C(\Ad_\rho)$, for any $\phi$ in $\lie{gl}(V)^\vee$, we 
get
\begin{equation}\label{action}
\pi_\R\left({\mathopen\langle}\phi,g{\mathclose\rangle}\right)
={\mathopen\langle}\phi,\pi_\lie{z}(g){\mathclose\rangle};
\end{equation}
note that this is a constant function on $H$.

\begin{theorem}\label{corollary}
Let $C(\Ad_\rho)$ the vector space of functions on $H$ generated by the matrix coefficients of 
$\Ad_\rho$, let $\Omega$ be a nonempty subset of $H$, write $\pi_\R$ for the equivariant 
projection of $C(\Ad_\rho)$ onto constant functions. 
Set 
\begin{equation} \label{eq:c}
c=\sup \{\pi_\R(f) | f\in C(\Ad_\rho), \sup_{\omega\in\Omega}\abs{f(w)}\leq1\}
\end{equation}
so that, for any $f$ in $C(\Ad_\rho)$, we have: 
\begin{equation} \label{eq:c-piR}
\sup_{\omega\in\Omega}\abs{f(\omega)}\cdot c\geq \abs{\pi_\R(f)}.
\end{equation}
Then
\begin{equation}\label{ineq}
\forall g\in \lie{gl}(V),\forall v \in V, 
\sup_{\omega\in\Omega} \Nm{\rho\left(\omega^{-1}\right) 
\cdot g \cdot \rho\left(\omega\right)(v)}\cdot c \geq \Nm{\pi_\lie{z}(g)(v)}.
\end{equation}
\end{theorem}
\begin{proof}
The remark \eqref{eq:c-piR} follows from the homogeneity of $\pi_\R$. 

Fix $g$ in $\lie{gl}(V)$, $v$ in $V$, and denote by $w$ the vector $\pi_\lie{z}(g)(v)$. Applying 
Cauchy-Schwarz inequality in $V$, we get, for any $\omega$ in $H$,
\begin{equation}\label{A}
\Nm{\rho\left(\omega^{-1}\right) \cdot g \cdot \rho\left(\omega\right)(v)}\cdot \Nm{w}
\geq \left(\rho\left(\omega^{-1}\right) \cdot g \cdot \rho\left(\omega\right)(v) \mid  w\right),
\end{equation}
where $(\cdot\mid \cdot)$ denotes the inner product on $V$. 
Note that the right-hand side, as a function of $\omega$, is a matrix coefficient belonging to 
$C(\Ad_\rho)$. Consequently, by definition of $c$,
\begin{equation}\label{B}
\sup_{\omega\in\Omega}\left(\rho\left(\omega^{-1}\right) \cdot g \cdot 
\rho\left(\omega\right)(v) \mid  w \right)\cdot c
\geq
\pi_\R
\Bigl(\left(\rho\left(\omega^{-1}\right) \cdot g \cdot \rho\left(\omega\right)(v) \mid w\right)\Bigr).
\end{equation}
Equation~\eqref{action} with $\phi:X\mapsto(X(v)\mid w)$ specializes in 
\begin{equation}\label{C}
\pi_\R\Bigl(\left(\rho\left(\omega^{-1}\right) \cdot g \cdot \rho\left(\omega\right)(v) \mid  w \right)
\Bigr)
=\left(\pi_\lie{z}(g)(v) \mid  w \right) =\Nm{w}^2.
\end{equation}
Applying $\sup_{\omega\in\Omega}$ to both sides of~(\ref{A}), combining with~(\ref{B}), 
substituting~(\ref{C}), we finally get 
\begin{equation}\label{D}\sup_{\omega\in\Omega} \Nm{\rho(\omega^{-1}) \cdot g \cdot 
\rho(\omega)(v)}\cdot\Nm{w}\cdot c\geq  \Nm{w}^2,\end{equation}
which implies (\ref{ineq}), as $\Nm{w}\geq0$.
\end{proof}

\begin{corollary}\label{constant} In the situation of Theorem \ref{corollary}, assume moreover that
\begin{enumerate}
 \item[$(*)$]\label{star} every matrix coefficient in $C(\Ad_\rho)$ that vanishes on $\Omega$ also 
vanishes on the whole $H$.
\end{enumerate}
Assuming $(*)$ and $\dim(V)>0$, we get $1\leq c\neq \infty$, and 
\begin{equation}\label{ineqP}
\forall p\in \exp_G(\lie{p}),\forall v \in V, \sup_{\omega\in\Omega} \Nm{\rho\left(\omega^{-1} \cdot 
p \cdot \omega\right)(v)} \geq \Nm{v}/c.
\end{equation}
\end{corollary}
\begin{proof}
 Assuming $\dim(V)>0$, constant functions are matrix coefficients of $\Ad_\rho$, hence by 
\eqref{eq:c},  $c\geq\pi_\R(1)=1$. Condition $(*)$ ensure that the map $f\mapsto \sup_{\omega
\in\Omega}\abs{f(w)}$ actually defines a \emph{norm}, instead of a mere semi-norm, on the 
subspace of $C(\Ad_\rho)$ on which it takes finite values. By \eqref{eq:c},  $c$~is the operator 
norm of the restriction to this subspace of the bounded linear application $\pi_\R$. Whence 
$c<\infty$.
 
 The inequality (\ref{ineqP}) follows from combining~\ref{corollary} and~\ref{effective}, and then 
dividing by $c>0$. 
\end{proof}

 \begin{remark} \label{rem:condstar} 
 \begin{enumerate} 
 \item Condition $(*)$ is satisfied for any Zariski dense subset $\Omega$ of $H$, and in 
particular\footnote{Recall $H$ is assumed to be connected, and being smooth, it is irreducible.} if 
$\Omega$ has nonempty interior or positive Haar measure. 
 \item If moreover $\Omega$ is bounded, then the map $f\mapsto \sup_{\omega\in\Omega}
\abs{f(w)}$ defines a norm on whole of $C(\Ad_\rho)$.  
 \item \label{itm:finite} Note that condition $(*)$ means that the evaluation maps $f\mapsto 
f(\omega)$, with $\omega$ in $\Omega$, generate the algebraic dual of $C(\Ad_\rho)$. 
Choosing a basis from this generating set, one can see that condition $(*)$ can still be met by 
replacing $\Omega$ by a subset of cardinality at most $\dim(C(\Ad_\rho))$. 
 \item Note that in terms of such a basis of $C(\Ad_\rho)$ (which can be obtained using a basis 
of $V$) one can effectively bound the constant $c$ from the above in Corollary~\ref{constant}.
 \end{enumerate}
\end{remark}

\section{Proof of Theorem \ref{lemma}}
We will show how to derive Theorem \ref{lemma} from Corollary~\ref{constant}. Actually we will 
establish the following more precise statement. The existence of $\Theta$ follows from 
criterion~(c) of \S\ref{conventions}. 

\begin{proposition}\label{proposition}In the situation of Theorem~\ref{lemma}, let $\Theta$ be a 
Cartan involution of $G$ under which $H$ is invariant, and let $K$ be the  maximal projectively 
compact subgroup of $G$ consisting of all the fixed points of $\Theta$. Write $\lie{g}$, 
$\lie{z}_G$, and $\lie{k}$ for the Lie algebras of $G$, $\Cent_G(H)$, and $K$ respectively. Let 
$\lie{p}=\left(\lie{k}+\lie{z}_G\right)^\bot$, the orthogonal complement  of the sum of $\lie{z}_G$ and 
$\lie{k}$ with respect to the Killing form on $\lie{g}$.

Then the subset $Y=K\cdot\exp_G(\lie{p})$ of $G$ satisfies both the conditions of 
Theorem~\ref{lemma}.
\end{proposition}
We will prove that $Y$ satisfies each of these conditions in the next two subsections. 
\subsection{The first condition}
We recall another theorem due to G.D.~Mostow.

\begin{theorem}[{Mostow~\cite[Theorem~5]{Mostow}}]\label{Mostow}
 Let $G$ be a connected semi-simple real Lie group and let $K$ be a maximal projectively  
compact subgroup of $G$. Let $\lie{g}$ denote the Lie algebra of $G$ and $\lie{k}$ the Lie 
algebra of $K$. Let $\lie{z}$ be any Lie subalgebra of $\lie{g}$. Orthogonality is understood with 
respect to the Killing form.

 Then the following application is a diffeomorphism.
\begin{equation}\label{Mostowmap}
\begin{array}{ccc}
K\times
\left(\lie{k}^\bot\cap(\lie{k}^\bot\cap\lie{z})^\bot\right)
\times\left(\lie{k}^\bot\cap\lie{z}\right) & \to & G \\
(k,Q,Z) & \mapsto & k\cdot \exp_G(Q) \cdot \exp_G(Z)
\end{array}
\end{equation}
\end{theorem}

Mostow states that $G$ ``decomposes topologically'', meaning that we have a homeomorphism. 
This is enough to establish the first condition of Theorem~\ref{lemma}, but we can verify this 
directly, as below, that map (\ref{Mostowmap}) is an immersion. As both sides 
of~(\ref{Mostowmap}) have equal dimension, (\ref{Mostowmap}) will be a local diffeomorphism, 
but being bijective, it will be an (analytic) diffeomorphism.
Let us prove that at each $(k,Q,Z)$ in 
$K\times\left(\lie{k}^\bot\cap(\lie{k}^\bot\cap\lie{z})^\bot\right)
\times\left(\lie{k}^\bot\cap\lie{z}\right)$ the tangent map is injective. 
\begin{proof}
Left and right translating one is reduced to the case where $Z=0$ and $k=\exp_G(0)$. Write $q=
\exp_G(Q)$, and let $dK$, $dQ$, and $dZ$ be arbitrarily small in $\lie{k}$, $\lie{p}$, and $\lie{k}^
\bot\cap\lie{z}$ respectively. 
 At first order,
 \begin{align*}\exp_G(dK)\exp_G(Q+{dQ})&\exp_G(dZ)\sim \\
 & q\cdot (q^{-1}\exp_G(dK)q)\exp_G(dQ)\exp_G(dZ).
 \end{align*}
 The latter equals $q\cdot\exp_G(\Ad_{q^{-1}}(dK))\exp_G(dQ)\exp_G(dZ)$, or, up to first order,
 $$q\cdot\exp_G(\Ad_{q^{-1}}(dK)+dQ+dZ).$$
  We will be done showing that $\Ad_{q^{-1}}(dK)+dQ+dZ$ cannot be zero for arbitrarily small 
and not simultaneously zero $dK$, $dQ$ and $dZ$, namely that $\Ad_{q^{-1}}(\lie{k})$, $
\left(\lie{k}^\bot\cap(\lie{k}^\bot\cap\lie{z})^\bot\right)$ and $\lie{k}^\bot\cap\lie{z}$ are in direct 
sum. Note that $\lie{k}$ and $\lie{k}^\bot$ are anisotropic of opposite sign (negative and positive 
resp.). By invariance of the Killing form, $\Ad(\exp_G(Q))(\lie{k})$ is negative 
(\cf~\ref{Killinginvariant}), hence has intersection $\{0\}$ with $\lie{k}^\bot$. Consequently 
$\Ad_{q^{-1}}(\lie{k})$ and $\lie{k}^\bot$ are in direct sum. As $\lie{k}^\bot$ is anisotropic, 
$\lie{z}\cap\lie{k}^\bot$ is supplementary to its orthogonal complement in $\lie{k}^\bot$.
 \end{proof}

Let $K$ and $Y$ be as in Proposition~\ref{proposition}. Applying Theorem~\ref{Mostow} to 
$\lie{z}=\lie{z_g}$, it follows that the equality $G=Y\cdot \Cent_G(H)^0$ is satisfied and that $Y$ 
defines a closed submanifold of $G$. In particular $Y$ satisfies condition $1$ of 
Theorem~\ref{lemma}.

\subsection{The second condition}
What is left, in order to prove Proposition \ref{proposition}, is  to show that $Y$ satisfies the 
condition~\eqref{ineqlemma} of the Theorem~\ref{lemma}. Fix $\rho$ as in 
Theorem~\ref{lemma}. We will prove~\eqref{ineqlemma} under a weaker hypothesis  on $\Omega$, namely the condition $(*)$ stated in Corollary~\ref{constant}.

\begin{proof}
If $\dim(V)=0$, then \eqref{ineqlemma} is immediate. So assume $\dim(V)>0$.

Note that it is enough to prove the inequality \eqref{ineqlemma} for any subset $\Omega_b$ of 
$\Omega$ instead of $\Omega$. Moreover, according to Remark~\ref{rem:condstar} 
(\ref{itm:finite}), we can assume this subset to be finite and still satisfy condition $(*)$. In 
particular such an $\Omega_b$ will be bounded.

Because $V$ is finite dimensional, all norms on $V$ are equivalent. Consequently, the validity of 
the inequality \eqref{ineqlemma} doesn't depend on the chosen norm on $V$, if one allows to 
change the constant. In particular one can assume that this norm is associated to a euclidean 
structure on $V$ as in~\ref{subsec:euclidean}. Then the corresponding inner product is $K$-
invariant.

Recall (proposition \ref{proposition}) that $Y=K\cdot \exp_{G}{\lie{p}}$. As the euclidean norm on 
$V$ is $K$ invariant, the inequality~\eqref{ineqlemma} for $y$ in $Y$ will follow from inequality~
\eqref{ineqlemma} for $y\in \exp_{G}(\lie{p})$. 

We only need to prove that there exists a constant $c>0$ such that
\begin{equation}\label{truc}
    \forall y\in \exp_{G}(\lie{p}),\,\forall v \in V,\quad 
    \sup_{\omega\in\Omega} \Nm{\rho(y \cdot \omega)(v)} \geq c\cdot\Nm{v}.
\end{equation}
As $\Omega$ is bounded, $C=\sup_{\omega\in\Omega_b}\NM{\rho(\omega^{-1})}$ is finite, 
where $\NM{\cdot}$ denotes the operator norm, and because $\dim(V)>0$, $C>0$. Because of 
the inequalities 
$$\Nm{\rho\left(\omega^{-1}\cdot y\cdot\omega\right)v}\leq\NM{\rho\left(\omega^{-1}\right)}\cdot
\Nm{\rho\left(y\cdot\omega\right)v}\leq C\cdot\Nm{\rho(y\cdot\omega)v},$$
equation~(\ref{truc}) follows from
\begin{equation}\label{truc2}
    \forall y\in \exp_{G}(\lie{p}),\,\forall v \in V,\quad \sup_{\omega\in\Omega} 
\Nm{\rho(\omega^{-1}\cdot y \cdot \omega)(v)} \geq Cc\cdot\Nm{v}.
\end{equation}
Let $c'$ be the constant given by \eqref{eq:c}. According Corollary~\ref{constant}, 
equations~(\ref{truc2}), and hence~(\ref{truc}) hold for $c=\frac{1}{Cc'}$.
\end{proof}

\begin{remark} \label{rem:explicit-c}
Proposition~\ref{proposition} and Theorem~\ref{lemma} are proved, with $Y$ given by 
Proposition~\ref{proposition} (or \ref{generalnotations}), assuming only that $\Omega$ satisfies 
condition $(*)$ of Corollary~\ref{constant}, and, whenever the norm on $V$ is given by 
\ref{subsec:euclidean}, with 
\begin{equation}\label{effectiveconstant}
c=\left(\sup_{\omega\in\Omega_b}\NM{\rho(\omega^{-1})}\cdot\sup \{\pi_\R(f) | f\in C(\Ad_\rho), 
\sup_{\omega\in\Omega_b}\abs{f(w)}\leq1\}\right)^{{-1}}.
\end{equation}
\end{remark}

\section{Subsequent enhancements: case of reductive~$G$}\label{general}
Actually, Theorem~\ref{lemma} can be generalized in different ways. First of all, if~$G$ is a 
semisimple linear Lie group, one can consider the algebraic structure (given by the algebra of 
matrix coefficients). In this case the Theorem~\ref{lemma} and its proof remain true if one only 
assume that the Zariski closure of~$H$ is Zariski connected, instead of $H$ being connected as 
a Lie group.

More importantly, in a different direction,
\begin{proposition} \label{prop:reductiveG}
The conclusion of Theorem~\ref{lemma} holds under the following relaxed conditions on $G$ 
and $H$: Let $G$ be a\/ \emph{reductive Lie group}; that is, $G$ has finitely many components 
and the adjoint action of $\lie{g}$ on itself is completely reducible. And Let~$H$ to be a 
connected\/ \emph{reductive subgroup of}~$G$; that is, the adjoint action of~$H$ on~$\lie{g}$ is 
completely reducible. 
\end{proposition}
\begin{proof}
Note that, for any compact subset~$C$ of~$G$, if we replace~$Y$ by~$CY$, the conclusion of 
Theorem~\ref{lemma} still hold, up to a change in the constant~$c$. This remark shows that 
without loss of generality we may assume that $G$ is connected.

 Let~$Z$ denote the center of~$G$ and~$[G,G]$ be the derived subgroup of~$G$. Set 
$H'=(HZ)\cap[G,G]$.  Note that~$H'$ is reductive in~$[G,G]$, because it has the same action as 
$H$ on~$\lie{g}$ and because~$[\lie{g},\lie{g}]$ is invariant subspace of~$\lie{g}$. We can apply 
Theorem~\ref{lemma} to~$H'$ in~$[G,G]$, in order to get a subset~$Y'$ of~$[G,G]$. As $
\Cent_{G}(H)=\Cent_{G}(HZ)=\Cent_{[G,G]}(H')Z$, we have~$Y'\Cent_{G}(H)=Y'\Cent_{[G,G]}
(H')Z=[G,G]Z=G$. Thus~$Y$ as a subset of~$G$ satisfies the first condition of 
Theorem~\ref{lemma}, with respect to~$G$ and~$H$. 

Let us check that~$Y$ also satisfies the second condition of Theorem~\ref{lemma}, namely 
formula~(\ref{ineqlemma}). First note we can replace $\Omega$ by a bounded subset, then that, 
for any bounded subset~$C$ of~$Z$, formula~(\ref{ineqlemma}) still holds, up to a change in 
constant, if we replace~$\Omega$ by $\Omega C$, and conversely. Consequently we can 
replace~$H$ by~$HZ$ and assume~$\Omega$ to have nonempty interior in~$HZ$. Taking a 
smaller subset we can assume $\Omega$, which we assumed to be bounded, is a product in~
$G$ of subsets of~$[G,G]$ and~$Z$. Using the converse above, we can replace~$HZ$ by~$H'$ 
and assume~$\Omega$, to be contained in~$H'$ and have nonempty interior in~$H'$.
\end{proof}

\section{Proof of Theorem~\ref{theofocusing}}

We now turn to the proof of Theorem~\ref{theofocusing}. We will in fact derive it from 
Theorem~\ref{lemma}. This was inspired by an argument of Kempf~\cite{Kempf} for reducing 
$S$-instability to instability; Kempf credits Mumford for the argument.

We consider $(y_n)_{n\in\N}$, $\Omega$, $\rho$ and $v$ as in the statement and prove the 
equivalence.

\begin{remark} \label{rem:AB}
\begin{enumerate}
\item The veracity of each of property~(A) and property~(B) is independent of the 
choice of the subset $\Omega$ of $H$, provided it is bounded and has nonempty interior in $H$.

\item Concerning property~(B), we first remark that $F$ depends only on the subspace $\langle
\Omega v\rangle$ generated by $\Omega v$. This space is contained in $\langle H v\rangle$, 
and not in any proper subspace. Indeed, $\Omega$ can not be contained in the inverse image 
by $h \to h v$ of a proper subspace of $\langle H v\rangle$: this inverse image is a proper 
differential subvariety, and has empty interior. We proved property~(B) depends only on 
$\langle H v\rangle$ and not on a specific $\Omega$.

\item \label{itm:A1}
Concerning property~(A), we first choose a basis of $\langle\Omega v\rangle$ from the 
generating subset $\Omega v$. Consider any bounded subset $\Omega^\prime$ of $H$. Then 
$\Omega^\prime v$ is bounded in $\langle H v\rangle$. Consequently, the coefficients of 
$\omega^\prime v$ in chosen basis will remain bounded as $\omega^\prime$ ranges over 
$\Omega^\prime$.  Write $(e_i)_{0\leq i \leq N}$ for the basis. The vector $y_n \omega^\prime v$ 
will be written with the same bounded coefficients in the basis $(y_n e_i)_{0\leq i \leq N}$ as 
$\omega^\prime v$ in the basis $(e_i)_{0\leq i \leq N}$. 
Let us now assume property~(A) for $\Omega$. As each $e_i$ belongs to $\Omega v$, the 
sequences $(y_n e_i)_{n \in \N}$ will be bounded. Consequently, the sequences 
$(y_n \omega^\prime v)_{n \in \N}$,  which are finite linear combinations of the formers, with 
uniformly bounded coefficients, are uniformly bounded, as $\omega^\prime$ ranges over 
$\Omega^\prime$. This proves property~(A) for $\Omega^\prime$.

\item \label{rem:A} From the previous argument, we deduce that 
property~(A) is equivalent to each of the following two variants:
\begin{equation} \tag{A'} 
\text{For each }\omega \text{ in }\Omega,\text{ the sequence }(y_n \omega v)_{n\in\N}
\text{ is bounded in }V,
\end{equation}
\begin{equation} \tag{A''} 
\text{For each }w \text{ in }\langle H v\rangle,\text{ the sequence }(y_n w)_{n\in\N}
\text{ is bounded in }V.
\end{equation}
\end{enumerate}
\end{remark}

\begin{proof}[Proof of $(B) \Rightarrow (A)$] This implication is the easiest one to prove and  
does not need the knowledge of $(y_n)_{n\in\N}$ being in $Y$, or $G$ being semisimple.

Let $F$ denote the point-wise stabilizer $\Omega v$. Assuming~(B), we know there is some 
compact set $C$ in $G$ such that each $y_n$ belongs to $C F$. The image subset $\rho(C)$ in 
$\End(V)$ is compact because $\rho$ is continuous. On the other hand $\Omega$ is bounded, 
hence contained in a compact, for instance $\overline{\Omega}$. Again, $
\rho(\overline{\Omega})$ is compact in $\End(V)$. 

For every $n$ in $\N$, and every $\omega$ in $\Omega$, one has 
\begin{equation} \label{formF}
 y_n \omega v \in \rho(C) \rho(F) \rho(\overline{\Omega}) v.
\end{equation}
But $F$ acts trivially on $\rho(\Omega) v$, hence on $\overline{\rho(\Omega) v}$ because the 
fixed point subspace of $F$ is closed. But, by continuity of $\rho$, the closure 
$\overline{\rho(\Omega) v}$ contains $\overline{\rho(\Omega)} v$. In equation~\eqref{formF} 
above, one can then forget about the action of $F$, which acts trivially on $
\rho(\overline{\Omega}) v$. It remains:
\begin{equation} \notag
\forall n\in\N,\forall\omega\in\Omega, y_n \omega v \in \rho(C)\rho(\overline{\Omega}) v.
\end{equation}
But $\rho(C)$ and $\rho(\overline{\Omega})$ are compact, and so is $\rho(C)
\rho(\overline{\Omega}) v$. This proves the sought for uniform boundedness of property~(A).
\end{proof}

The second implication is more involved. To summarize our approach a few words, we first 
convert boundedness in~(A) into convergence in some auxiliary representation space (after 
passing to a subsequence). We then employ the clever idea due to Kempf for passing from $S$-
instability to $\{0\}$-instability using the following easily provable Lemma~1.1(b) of \cite{Kempf},  
to convert the situation of convergence toward any vector into convergence to $0$ in another 
auxiliary finite dimensional representation space. As the property property~(A) carries over to the 
new representation, we will be able to reduce everything to Theorem~\ref{lemma}.

\begin{proof}[Proof of  {\rm (A)} $\Rightarrow$ {\rm (B)}] {\em Assume by contradiction that {\rm (A)} holds, 
but not {\rm (B)}}. Consider the space of functions from $\Omega$ to $V$, and more specifically the 
$G$-invariant subspace $W$ generated by the function $f_v:\omega \mapsto \omega v$. We can 
realize $W$ as a $G$-subspace of  the finite dimensional space $\Hom(\langle f_v(\Omega) 
\rangle,V)$, where $G$-acts on the image space. Note that $f_{v}$ corresponds to the identity 
homomorphism of $\langle f_{v}(\Omega)\rangle$ in $V$. 

The $G$~action on $f_v$ is via point-wise translation on the values of $f_v$. These values span 
$\Omega v$. The stabilizer in $G$ of the function $f_v$ is then the point-wise stabilizer of 
$\Omega v$. We denote this stabilizer by $F$.

Since we assumed that (B) fails to hold, the sequence $(y_n)_{n\in\N}$ is not  bounded in $G$ 
modulo $F$ on the right. By passing to a subsequence, one may assume no subsequence of the 
sequence $(y_n)_{n\in\N}$ is  bounded in $G$ modulo $F$ on the right.

Let us assume that property~(A) holds (it then holds for any subsequence). In other words 
property~(A) tells that the vector $y_n f_v$ of $W$ can be bounded independently of $n$. 
Replacing  by a subsequence, one may assume that the sequence $(y_n f_v)_{n \in \N}$ is 
convergent in the finite dimensional vector space $W$. Let $f_\infty$ be its limit.

We claim that the limit $f_\infty$ can not belong to the orbit $G f_v$. By contradiction, if 
$(y_n f_v)_{n\in\N}$ were converging inside the orbit $G v$, then its inverse image under the bijective 
map
\begin{equation}\label{orbitmap}
G/F \rightarrow G f_v,~g F \longmapsto g f_v
\end{equation}
would be convergent in $G/F$, hence would be bounded in $G/F$, contradicting the failure of 
property~(B). For this argument to work, we have to know that the inverse map of~\eqref{orbitmap} is continuous. By~\cite[Corollary~2 of Lemma~3.2]{PR}, $Gf_{v}$ is open in its closure, 
and in particular it is locally compact.  Therefore as a consequence of the Baire's category theorem for 
locally compact second countable spaces, the orbit map~\eqref{orbitmap} is open, and hence a 
homeomorphism. 

This limit $f_\infty$ then belongs to $\overline{G f_v}\smallsetminus G f_v$, which we denote by 
$\partial(Gv)$.

Let $h$ be in $H$. Then by Remark~\ref{rem:AB}-(\ref{rem:A}), property~(A) holds also for $h \Omega$ 
instead 
of $\Omega$. The veracity of Property~(B) is clearly untouched by substituting $(y_n)_{n \in \N}$ 
with $(y_n h)_{n \in \N}$. Arguing with function $h f_v:\omega \mapsto h \omega v$ instead of 
$f_v$, we conclude the sequence $(y_n h f_v)_{n \in \N}$ has a limit, say $f_\infty^h$ in $W$, 
and that this limit belongs to $\overline{G h f_v}\smallsetminus G h f_v$, which equals 
$\overline{G f_v}\smallsetminus G f_v$.

Let $\Zar\partial(G f_v)=\Zcl(G f_v)\smallsetminus\Zcl(G)f_{v} $ denote the boundary of $Gf_{v}$ with 
respect to the Zariski topology; see Lemma~\ref{bordzariski} stated below. From the closed orbit 
lemma \cite[Proposition~2.23]{PR}, one knows that $S:=\Zar\partial(G f_v)$ is a Zariski closed 
($G$-invariant) subset in $W$. By  \cite[Lemma~1.1(b)]{Kempf}, over $k=\R$, there exists a finite 
dimensional linear $G$-space $W'$ and a $G$-equivariant polynomial map
$$ \Phi: W \rightarrow W^\prime$$ such that $\Phi^{-1}(0)=S$. Clearly, $f_v\not\in S$, therefore 
$\Phi(f_{v})\neq 0$. On the other hand, for each $h\in H$, in view of the above observation and by 
Lemma~\ref{bordzariski},  $f_\infty^{h}\in S$, and hence $\Phi(f_\infty^h)=0$.

Let us recall the situation. We have a sequence $(y_n)_{n\in\N}$ in $Y$ such that, for each $h$ 
in $H$, the sequence $(y_n h \Phi(f_v))_{n \in \N}$ converges to $0$ whereas $h \Phi(f_v)$ is 
never zero. We now consider a compact subset $\Omega^\prime$ of $H$ with nonempty interior. 
This can be found because $H$ is a connected Lie group. Then the sequence of (continuous) 
functions $(h \mapsto y_n h \Phi(f_v))_{n \in \N}$ is converging point-wise to $0$, hence, on 
$\Omega^\prime$, is uniformly converging to $0$ by the argument as in 
Remark~\ref{rem:AB}-(\ref{itm:A1}). 

In particular, the subsets $y_n \Omega^\prime \Phi(f_v)$ of $W^\prime$ are uniformly 
converging to $0$ as $n$ goes to $\infty$. But this contradicts Theorem~\ref{lemma} applied to
\begin{enumerate}\renewcommand{\labelenumi}{(\roman{enumi})}
\item the bounded subset $\Omega^\prime$ of $H$ with nonempty interior;
\item the representation of $G$ on $W^\prime$;
\item any norm on $W^\prime$.
\end{enumerate}
This theorem says $y_n \Omega^\prime \Phi(f_v)$ can not be bounded above, in the chosen 
norm, by $c \Nm{\Phi(f_v)}$.

This contradiction completes the proof of Theorem~\ref{theofocusing}, modulo the following 
lemma.
\end{proof}

\begin{lemma} \label{bordzariski} Consider a $\R$-vector space $V$, and let a $G$ be a 
connected real Lie subgroup of $\GL(V)$. pick any $v$ in $V$.
We denote
\begin{itemize}
\item by $\Zcl(G)$ the Zariski closure of $G$ in  $GL(V)$,
\item by $\Zcl(G v)$ the Zariski closure of the orbit $Gv$ in $V$,
\item by $\partial(Gv)=\overline{G v}\smallsetminus G v$, the boundary for the metric topology, 
and
\item by $\Zar\partial(Gv)=\Zcl(G v)\smallsetminus \Zcl(G) v$, the boundary for the Zariski 
topology.
\end{itemize}
Assume now that $G$ is open in $\Zcl(G)$, (which means that $G$ is a real algebraic Lie group; 
see~\cite[Theorem~3.6 and Corollary~1]{PR},) for instance if $G$ is semisimple.

Then $\Zar\partial(Gv)\cap \overline{G v}=\partial(Gv)$: one has 
$\partial(Gv)\subseteq\Zar\partial(Gv)$ and $\Zar\partial(Gv)\cap G.v=\emptyset$. 
\end{lemma}
\begin{proof}
Firstly, as $G v$ is contained in $\Zcl(G v)$, which does not meet $\Zar\partial(Gv)$, one gets 
easily: $\Zar\partial(Gv)\cap Gv=\emptyset$. 

Since $G$ is open in $\Zcl(G)$, for any point $w\in \Zcl(G) v$ we have that $G w$ is open in 
$\Zcl(G) w$ (the orbit map $\Zcl(G)\to \Zcl(G) v$ is an open map, 
by~\cite[Corollary~2 of Lemma~3.2]{PR}). Using the 
Closed orbit lemma~(\cite[\S I.1.8]{BorelLAG}, or~\cite[Proposition~2.23]{PR}), 
we know that $\Zcl(G) v$ is Zariski open in $\Zcl(G v)$, hence 
open. It follows that $G w$, for $w$ in $\Zcl(G) v$, is also open in $\Zcl(G v)$. We also know that 
$\Zcl(G v)$ is Zariski closed, hence closed, and contains $G v$; it hence contains 
$\overline{G v}$. Therefore $G w\cap\overline{G v}$ is an open subset of $\overline{Gv}$. 

Consider now a point $x$ of $\partial(Gv)$. Then the $G$-orbit $G x$ of $x$ is distinct from 
$G v$. But $x$ belongs to $\overline{G v}$, and can be approached along $G v$. Hence any 
neighborhood of $x$ meets at least two $G$ orbits: $G v$ and $G x$. It follows that $G x$ is not 
open in $\overline{G v}$ at $x$, hence not open in $\overline{G v}$.

Therefore $x$ can not be of the form $w$, with $w\in\Zcl(G) v$. Now $\partial(Gv)$ is contained 
in $\overline{Gv}$, hence in $\Zcl(G v)$, while it does not meet $\Zcl(G) v$. in other words, 
$\partial(Gv)\subseteq\Zar\partial(Gv)$. 
\end{proof}

\begin{remark} \label{rem:boundary} The proof of Lemma~\ref{bordzariski} uses only that the 
metric topology is finer than the Zariski topology, and that the orbit map $\Zcl(G)\to \Zcl(G) v$ is 
an open map (\cite[Corollary~2 of Lemma~3.2]{PR}). The lemma and the proof then hold  for any 
non-discrete locally compact field $k$ of characteristic zero (hypotheses from \cite[Sec.~3.1]
{PR}), for any group $G$ which is open in $\Zcl(G)(k)$.
\end{remark}

\begin{remark} \label{rem:kLC} Our proof of Theorem~\ref{theofocusing} is essentially algebraic 
in nature. It can be transposed mutatis mutandis to other locally compact fields $k$, provided: 
(1)~That one has the analog of Lemma~\ref{bordzariski}, see Remark~\ref{rem:boundary}; (2)~one 
has the analog of Theorem~\ref{lemma} (as, for instance, in~\cite{These}).
\end{remark}


\begin{thebibliography}{GMO08}

\bibitem[Bor91]{BorelLAG}
Armand Borel.
\newblock {\em Linear algebraic groups}, volume 126 of {\em Graduate Texts in
  Mathematics}.
\newblock Springer-Verlag, New York, second edition, 1991.

\bibitem[Bou60]{B}
N.~Bourbaki.
\newblock {\em \'{E}l\'ements de math\'ematique. {XXVI}. {G}roupes et
  alg\`ebres de {L}ie. {C}ha\-pi\-tre~1: {A}lg\`ebres de {L}ie}.
\newblock Actualit\'es Sci. Ind. No. 1285. Hermann, Paris, 1960.

\bibitem[Bou72]{B3}
N.~Bourbaki.
\newblock {\em \'{E}l\'ements de math\'ematique. {F}asc. {XXXVII}. {G}roupes et
  alg\`ebres de {L}ie. {C}hapitre {II}: {A}lg\`ebres de {L}ie libres.
  {C}hapitre {III}: {G}roupes de {L}ie}.
\newblock Hermann, Paris, 1972.
\newblock Actualit{\'e}s Scientifiques et Industrielles, No. 1349.

\bibitem[CLT10]{ACLT}
Antoine Chambert-Loir and Yuri Tschinkel.
\newblock Igusa integrals and volume asymptotics in analytic and adelic
  geometry.
\newblock {\em Confluentes Math.}, 2(3):351--429, 2010.

\bibitem[Dan79]{Dani-nondiv}
S.~G. Dani.
\newblock On invariant measures, minimal sets and a lemma of {M}argulis.
\newblock {\em Invent. Math.}, 51(3):239--260, 1979.

\bibitem[Dan84]{Dani-rk2}
S.~G. Dani.
\newblock On orbits of unipotent flows on homogeneous spaces.
\newblock {\em Ergodic Theory Dynam. Systems}, 4(1):25--34, 1984.

\bibitem[DM90]{DM-MathAnn}
S.~G. Dani and G.~A. Margulis.
\newblock Orbit closures of generic unipotent flows on homogeneous spaces of
  {${\rm SL}(3,{\bf R})$}.
\newblock {\em Math. Ann.}, 286(1-3):101--128, 1990.

\bibitem[DM91]{DM-asymptotic}
S.~G. Dani and G.~A. Margulis.
\newblock Asymptotic behaviour of trajectories of unipotent flows on
  homogeneous spaces.
\newblock {\em Proc. Indian Acad. Sci. Math. Sci.}, 101(1):1--17, 1991.

\bibitem[DM93]{DM-Limit}
S.~G. Dani and G.~A. Margulis.
\newblock Limit distributions of orbits of unipotent flows and values of
  quadratic forms.
\newblock In {\em I. {M}. {G}el\cprime fand {S}eminar}, volume~16 of {\em Adv.
  Soviet Math.}, pages 91--137. Amer. Math. Soc., Providence, RI, 1993.

\bibitem[DRS93]{DRS}
W.~Duke, Z.~Rudnick, and P.~Sarnak.
\newblock Density of integer points on affine homogeneous varieties.
\newblock {\em Duke Math. J.}, 71(1):143--179, 1993.

\bibitem[DS84]{DS}
S.~G. Dani and John Smillie.
\newblock Uniform distribution of horocycle orbits for {F}uchsian groups.
\newblock {\em Duke Math. J.}, 51(1):185--194, 1984.

\bibitem[EM93]{EM}
Alex Eskin and Curt McMullen.
\newblock Mixing, counting, and equidistribution in {L}ie groups.
\newblock {\em Duke Math. J.}, 71(1):181--209, 1993.

\bibitem[EMS96]{EMS}
A.~Eskin, S.~Mozes, and N.~Shah.
\newblock Unipotent flows and counting lattice points on homogeneous varieties.
\newblock {\em Ann. of Math. (2)}, 143(2):253--299, 1996.

\bibitem[EMS97]{EMSGAFA}
A.~Eskin, S.~Mozes, and N.~Shah.
\newblock Non-divergence of translates of certain algebraic measures.
\newblock {\em Geom. Funct. Anal.}, 7(1):48--80, 1997.

\bibitem[GMO08]{GMO}
Alex Gorodnik, Fran{\c{c}}ois Maucourant, and Hee Oh.
\newblock Manin's and {P}eyre's conjectures on rational points and adelic
  mixing.
\newblock {\em Ann. Sci. \'Ec. Norm. Sup\'er. (4)}, 41(3):383--435, 2008.

\bibitem[GO11]{GO:Rational}
Alex Gorodnik and Hee Oh.
\newblock Rational points on homogeneous varieties and equidistribution of
  adelic periods.
\newblock {\em Geom. Funct. Anal.}, 21(2):319--392, 2011.
\newblock With an appendix by Mikhail Borovoi.

\bibitem[Hel78]{Helgason}
Sigurdur Helgason.
\newblock {\em Differential geometry, {L}ie groups, and symmetric spaces},
  volume~80 of {\em Pure and Applied Mathematics}.
\newblock Academic Press Inc. [Harcourt Brace Jovanovich Publishers], New York,
  1978.

\bibitem[Kem78]{Kempf}
George~R. Kempf.
\newblock Instability in invariant theory.
\newblock {\em Ann. of Math. (2)}, 108(2):299--316, 1978.

\bibitem[KM98]{KM-Annals}
D.~Y. Kleinbock and G.~A. Margulis.
\newblock Flows on homogeneous spaces and {D}iophantine approximation on
  manifolds.
\newblock {\em Ann. of Math. (2)}, 148(1):339--360, 1998.

\bibitem[KT07]{KT}
Dmitry Kleinbock and George Tomanov.
\newblock Flows on {$S$}-arithmetic homogeneous spaces and applications to
  metric {D}iophantine approximation.
\newblock {\em Comment. Math. Helv.}, 82(3):519--581, 2007.

\bibitem[Mar75]{Mar-nondiv}
G.~A. Margulis.
\newblock On the action of unipotent groups in the space of lattices.
\newblock In {\em Lie groups and their representations ({P}roc. {S}ummer
  {S}chool, {B}olyai, {J}\'anos {M}ath. {S}oc., {B}udapest, 1971)}, pages
  365--370. Halsted, New York, 1975.

\bibitem[Mos55a]{Mostowherm}
G.~D. Mostow.
\newblock Self-adjoint groups.
\newblock {\em Ann. of Math. (2)}, 62:44--55, 1955.

\bibitem[Mos55b]{Mostow}
G.~D. Mostow.
\newblock Some new decomposition theorems for semi-simple groups.
\newblock {\em Mem. Amer. Math. Soc.}, 1955(14):31--54, 1955.

\bibitem[PR94]{PR}
Vladimir Platonov and Andrei Rapinchuk.
\newblock {\em Algebraic groups and number theory}, volume 139 of {\em Pure and
  Applied Mathematics}.
\newblock Academic Press Inc., Boston, MA, 1994.

\bibitem[Ra91]{R-measure}
Marina Ratner.
\newblock On {R}aghunathan's measure conjecture.
\newblock {\em Ann. of Math. (2)}, 134(3):545--607, 1991.

\bibitem[Ric09]{These}
Rodolphe Richard.
\newblock {S}ur {Q}uelques questions d'\'{e}quidistribution en
  g\'{e}om\'{e}trie arithm{\'e}tique.
\newblock \url{http://tel.archives-ouvertes.fr/docs/00/43/85/15/PDF/\newline
  Rodolphe-Richard-These-fois12-v2.pdf}, 2009.

\bibitem[RZ16]{focusing}
Rodolphe Richard and Thomas Zamojski.
\newblock Limit distribution of Translated pieces of possibly irrational leaves in S-arithmetic 
homogeneous spaces
\newblock {\em eprint arXiv:1604.08494}, 52 pages.

\bibitem[Sha91]{Shah-uniform}
Nimish~A. Shah.
\newblock Uniformly distributed orbits of certain flows on homogeneous spaces.
\newblock {\em Math. Ann.}, 289(2):315--334, 1991.

\bibitem[Sha96]{Shah-horo}
Nimish~A. Shah.
\newblock Limit distributions of expanding translates of certain orbits on
  homogeneous spaces.
\newblock {\em Proc. Indian Acad. Sci. Math. Sci.}, 106(2):105--125, 1996.

\bibitem[Sha09]{Shah-son1}
Nimish~A. Shah.
\newblock Limiting distributions of curves under geodesic flow on hyperbolic
 manifolds.
\newblock{\em Duke Math. J.},  148(2):251--279, 2009.

\bibitem[Tom00]{Tom-arithmetic}
George Tomanov.
\newblock Orbits on homogeneous spaces of arithmetic origin and approximations. In: Analysis 
on homogeneous spaces and representation theory of Lie groups,
 OkayamaÐKyoto (1997).
  \newblock {\em Adv. Stud. Pure Math.}, 26:265--297, 2000. 
  
\end{thebibliography}

\def\cprime{$'$}

\end{document}